\documentclass[12pt]{amsart}
\usepackage{latexsym,enumerate}
\usepackage{amssymb, xcolor,mathrsfs}
\usepackage{amsmath,amsthm,amsfonts,amssymb,latexsym, mathabx}
\usepackage{comment}

\newtheorem{theorem}{Theorem}
\newtheorem{lemma}[theorem]{Lemma}
\newtheorem{proposition}[theorem]{Proposition}
\newtheorem{corollary}[theorem]{Corollary}

\newtheorem{question}[theorem]{Question}

\theoremstyle{definition}

\newcommand{\Z}{\mathbb{Z}}
\newcommand{\Q}{\mathbb{Q}}

\newcommand{\C}{\mathbb{C}}

\newcommand{\Cent}{\mathbf{Z}}

\newcommand{\ol}{\overline}

\newcommand{\Mod}[1]{\ (\mathrm{mod}\ #1)}

\DeclareMathOperator{\SL}{SL}

\DeclareMathOperator{\PSL}{PSL}
\DeclareMathOperator{\Sz}{Sz}

\DeclareMathOperator{\Syl}{Syl}

\DeclareMathOperator{\Irr}{Irr}

\begin{document}
	
	\title[On Generalized Characters]{On Generalized Characters Whose Values on Nonidentity Elements are Sums of at Most Two Roots of Unity}

	\author[Christopher Herbig]{Christopher Herbig}
	\address{%
		Northern Illinois University\\
		Department of Mathematical Sciences\\
		Dekalb, IL 60115\\
		USA}
	\email{cherbig@niu.edu}

	\subjclass[2020]{20C15, 20K01 (Primary) 20B05, 05C25 (Secondary)}
	\keywords{generalized characters, abelian groups, finite groups, prime graphs}

	\date{\today}
	
	\maketitle
	
	
	\begin{abstract}
	A character of a finite group having degree $n$ takes values which may be expressed as sums of $n$ or fewer roots of unity. In this note, we prove a result which describes the irreducible constituents of generalized characters on abelian groups whose values on nonidentity elements are expressible as sums of two or fewer roots of unity. In \S4, we apply our main result to obtain information about the connectivity of prime graphs for groups admitting such characters.
	\end{abstract}
	
	
	\section{Introduction}
	
	All groups are assumed to be finite, and our notation follows that of \cite{Is}. In particular, all characters considered will be complex valued, and we denote the set of irreducible characters of a group $G$ as $\Irr(G)$. By a character, we mean any sum of irreducible characters, and by a generalized character, we mean any $\Z$-linear combination of irreducible characters of $G$. We will always take $\zeta_n$ to be the primitive $n$-th root of unity $\exp(2i\pi/n)$, and we let $\Q_n$ denote the cyclotomic extension of $\Q$ generated by a primitive $n$-th root of unity. For a group $G$, we take $G^{\#}$ to be the set of all nonidentity elements of $G$. We let $\pi(G)$ denote the set of all prime divisors of $G$, and for $\pi \subset \pi(G)$, we let $\pi'$ denote $\pi(G) - \pi$. If $n$ is any integer and $\pi$ is any set of primes, then $n_\pi$ is taken to denote the largest divisor of $n$ whose prime divisors all lie in $\pi$. Further, if $\pi = \{p\}$ for some prime $p$, we simply write $n_p$ as opposed to $n_{\{p\}}$. Similarly, if $x \in G$, then for any $\pi \subset \pi(G)$, we may uniquely factorize $x$ as a product of a $\pi$-element $x_\pi$ and a $\pi'$-element $x_{\pi'}$ such that $x = x_\pi x_{\pi'} = x_{\pi'} x_\pi$.  Lastly, we take $\rho_G$ to denote the regular character of $G$.
	
	Our work here is largely motivated by similar work of Geoffrey Robinson. Motivated in part by the Glauberman correspondence, Robinson has studied generalized characters of groups whose values on nonidentity elements are roots of unity. In particular, the main result of \cite{R} provides a purely group-theoretic characterization of groups admitting such a generalized character, working under the additional assumption that the Sylow 2-subgroups of the group have no cyclic subgroups of index 2. For reference, we partially state Robinson's main result below:
	
	\begin{theorem}[Robinson, 2011]\label{T:Rob1}
		Assume that the Sylow 2-subgroups of a finite group $G$ have no cyclic subgroups of index 2. Let $\chi$ be a generalized character taking root of unity values on $G^{\#}$. There exists a subset $\pi$ of $\pi(G)$ such that the restriction of $\chi$ to any $\pi$-subgroup [resp. $\pi'$-subgroup] $H$ takes the form $a\rho_H + \lambda$ [resp. $a\rho_H - \lambda$] for some $a \in \Z$ and $\lambda \in \Irr(H)$. 
		Further, if $x \in G^{\#}$ is neither a $\pi$-element nor a $\pi'$-element, then either $x_\pi$ or $x_{\pi'}$ is an involution.
		
		Conversely, assume that the Sylow 2-subgroups of $G$ have no cyclic subgroup of index 2. Further, assume there exists $\pi \subseteq \pi(G)$ with the following property: For each $x \in G^\#$ which is neither a $\pi$-element nor a $\pi'$-element, either $x_\pi$ or $x_{\pi'}$ is an involution. There then exists a generalized character $\chi$ of $G$ such that $\chi(x) = 1$ [resp. $-1$] for each nonidentity $\pi$-element [resp. $\pi'$-element].
	\end{theorem}
	
	Much earlier in \cite{R2}, Robinson made the following observation.
	
	\begin{theorem}[Robinson, 1999]\label{T:Rob2}
		Let $\chi$ be a generalized character of a group $G$ such that $\chi(x)$ is either zero or a root of unity for each $x \in G^{\#}$, then exactly one of the three following cases occurs:
		\begin{itemize}
			\item[(i)] $\chi$ is an integer multiple of the regular character $\rho_G$.
			\item[(ii)] $\chi(x)$ is a root of unity for all $x \in G^{\#}$. In particular, if the Sylow 2-subgroups of $G$ have no cyclic subgroup of index 2, then the main result of \cite{R} applies.
			\item[(iii)] There exists a nonempty, proper subset $\pi$ of $\pi(G)$ such that every element of $G^{\#}$ is either a $\pi$-element or a $\pi'$-element. 
		\end{itemize}
		Moreover, if $G$ is any group satisfying the condition in (iii), there exists a generalized character of $G$ whose values on $\pi$-elements [resp. $\pi'$-elements] of $G^{\#}$ are 0 [resp. 1]. 
	\end{theorem}
	
	The third case above can be rephrased in terms of the prime graph of $G$, denoted $\Gamma(G)$. The vertex set of this graph is given by $\pi(G)$, and two vertices, $p$ and $q$, of $\pi(G)$ are linked in $\Gamma(G)$ iff there exists an element of order $pq$.  We obtain the following equivalent rephrasing as an immediate consequence of the definition of $\Gamma(G)$:
	
	\begin{corollary}\label{C:3}
		Let $G$ be a group. The prime graph $\Gamma(G)$ is disconnected if and only if there exists a generalized character $\chi$ whose values on $G^{\#}$ are either zero or roots of unity where $\chi$ is not of the form given in cases (i) and (ii) in the above theorem. 
	\end{corollary} 
	
	Groups having a disconnected prime graph have rather restricted structure. Using the classification of finite simple groups, results from \cite{GK} and \cite{W} imply the following: If $\Gamma(G)$ is disconnected, then either $G$ is a Frobenius group, a 2-Frobenius group, or else has exactly one nonabelian composition factor.

	Theorem \ref{T:Rob2} is also applicable to certain classes of permutation groups. A permutation group $G$ acting on a set $\Omega$ is said to have \emph{minimal degree} $f$ if each $g \in G^{\#}$ fixes at most $|\Omega| - f$ points in $\Omega$. The transitive permutation groups of minimal degree $|\Omega|$ are precisely those which act regularly. Those of minimal degree $|\Omega| - 1$ include the class of Frobenius groups. Those of minimal degree $|\Omega| - 2$ include the class of Zassenhaus groups, i.e., doubly transitive groups having no nonidentity element fixing three points simultaneously. These Zassenhaus groups include the groups $\PSL(2,q)$ for $q$ odd along with the Suzuki groups, $\Sz(q)$.

	Now, assume that $G$ acts on $\Omega$ with minimal degree at least $|\Omega| - 2$. Then, $G$ has a permutation character of the form $1 + \chi$, where $1$ denotes the trivial character of $G$ and $\chi$ takes values in the set $\{-1,0,1\}$.  With this in mind, we obtain the following as a consequence of Theorem \ref{T:Rob2}.
	\begin{corollary}\label{C:4}
		Let $G$ be a transitive permutation group acting on $\Omega$ with minimal degree $|\Omega| - 2$. One of the following occurs:
		\begin{itemize}
			\item[(i)] There exists a nontrivial, proper subset $\pi$ of $\pi(G)$ such that every element of $G^{\#}$ is either a $\pi$-element or a $\pi'$-element. In particular, $\Cent(G) = 1$.
			\item[(ii)] $|G|$ is even, and there is no $g \in G^{\#}$ leaving exactly one point in $\Omega$ fixed. 
		\end{itemize}
	\end{corollary}
	
	\begin{proof}
		Let the permutation character of the action be denoted by $1 + \chi$ for some character $\chi$. The character $\chi$ takes values in the set $\{-1,0,1\}$ as mentioned above, so Theorem \ref{T:Rob2} is applicable. It follows as a consequence of the orbit counting lemma that $\chi$ takes the value $-1$ on at least one element of $G^{\#}$, so $\chi$ is not a multiple of the regular character. 
		
		If $\chi$ is nonzero on $G^{\#}$, then it follows that there exists no $g \in G^{\#}$ fixing exactly one element of $\Omega$. Thus, $\chi$ is of the form considered in Theorem \ref{T:Rob1}, so if the first case of this corollary does not occur, Theorem \ref{T:Rob1} implies that $G$ has a nontrivial Sylow 2-subgroup, and in particular, $|G|$ is even.
		
		The only case left unconsidered is the case where $\chi$ takes a mixture of zero and root of unity values on $G^{\#}$, but this is precisely the situation of case (iii) in Theorem \ref{T:Rob2}, so the result holds.
	\end{proof}
	
	Permutation groups acting with relatively large minimal degree have received extensive study for their applications to Riemann surfaces, and much is known about the structure of such groups (cf. \cite[Theorem 1.1]{HW}). In particular, the Zassenhaus groups provide us with many irreducible examples of the characters considered in case (iii) of Theorem \ref{T:Rob2}, since in this case, the group acts doubly transitively, and so, the constituent $\chi$ in the permutation character $1 + \chi$ is irreducible. Further, examples of that described in case (ii) in the above can occur and with nontrivial center. For instance, consider the transitive action of $\SL(2,3)$ on the cosets of one of its Sylow 3-subgroups. One verifies that under this action, the constituent $\chi$ of the permutation character $1 + \chi$ takes values $\pm 1$ off the identity, and so, no element fixes exactly one point.
	
	We wonder more generally about generalized characters whose values on nonidentity elements are sums of at most $k$ roots of unity where $k$ is relatively small. For the purpose of this note, we shall only be concerned with the case where $k = 2$. Our first main result in this regard concerns abelian groups. 
	
	\begin{theorem}[Main Theorem]\label{T:B}
		Let $G$ be an abelian group having a generalized character $\chi$ whose values on $G^{\#}$ are sums of at most 2 roots of unity. Then, we have that $\chi = a\rho_G + \delta_1\lambda_1 + \delta_2\lambda_2$ for  $a \in \Z$, $\delta_i \in \{-1,0,1\}$, and $\lambda_i \in \Irr(G)$, except possibly in the case where either $2, 3,$ or $5 \in \pi(G)$. 
		
		If $|G|$ is divisible by either 2,3, or 5, $\chi$ may take one of four possible forms:
		\begin{itemize}
			\item[(i)] $\chi = a\rho_G \pm (\lambda_1 + \lambda_2 + \lambda_3)$,
			\item[(ii)] $\chi = a\rho_G \pm (\lambda_1 + \lambda_2 - \lambda_3)$,
			\item[(iii)] $\chi = a\rho_G \pm (\lambda_1 + \lambda_2 + \lambda_3 + \lambda_4)$, or
			\item[(iv)] $\chi = a\rho_G \pm (\lambda_1 + \lambda_2 + \lambda_3 - \lambda_4)$,
		\end{itemize}
		for some $a \in \Z$ and distinct $\lambda_i \in \Irr(G)$, except possibly in the case that $|G| \leq 21$, where we might have that $\chi = a\rho_G \pm \sum_{i=1}^k \lambda_i$ for distinct $\lambda_i \in \Irr(G)$ and $5 \leq k \leq 7$. 
	\end{theorem}
	
	Given any generalized character whose values on nonidentity elements are sums of few roots of unity, we would intuitively expect that the generalized character would not differ too much from the regular character, and this result shows this to usually be the case under these circumstances, provided that the prime divisors of the group order are sufficiently large. We also mention that we have examples of each of the four outlier cases, showing that we do indeed need to make exceptions for these cases. For instance, in the case of (iv), take $G$ to be the cyclic group of order 12 and let the $\lambda_i$ be the four linear characters (in any order) whose kernels contain the subgroup of order 3. The cyclic group of order 15 has a counterexample of the type in case (i) and of the type in case (ii) as well. An example of type (iii) is also found in the cyclic group of order 9. Further, the requirement that $|G| \leq 21$ cannot be sharpened. By employing GAP \cite{Gap}, we found that one can choose five distinct linear characters of the cyclic group of order 21 whose values on the nonidentity elements are sums of two roots of unity. Similarly, there exists a sum of seven distinct irreducible characters of the cyclic group of order 15 satisfying this property. 
	
	The fact that the situation is somewhat controlled for abelian groups in the above theorem makes us wonder whether we can obtain a result analogous to Theorem 2. One direction turns out to be a fairly straightforward application of Brauer's Characterization of Characters. In fact, we prove something stronger.
	
	\begin{proposition}\label{P:6}
		If $G$ is a group and $\Gamma(G)$ has $k$ connected components, there exists $\chi \in \Z[\Irr(G)]$ whose set of values on $G^{\#}$ is precisely $\{a_i\}_{i=0}^{k-1}$ where the $a_i$ are distinct integers.
	\end{proposition}
	
	\begin{proof}
		Denote the $k$ connected components of the prime graph as $\Gamma_i(G)$ for $i=0,\ldots,k-1$, and let each of the respective vertex sets be denoted as $\pi_i$. Every element of $G^{\#}$ is a $\pi_i$-element for some $i$. For a $\pi_i$-element $x \in G^{\#}$, define $\chi: G^{\#} \to \C$ by setting $\chi(x) = a_i$. It is clear that $\chi$ is invariant on the conjugacy classes of $G^{\#}$. It remains do determine a valid value that $\chi(1)$ can take.
		
		Now, by the Chinese Remainder Theorem, there exists an integer $a$ such that $a \equiv a_i \Mod{|G|_{\pi_i}}$ for each $i$. Define $\chi(1) = a$. Now, the restriction of $\chi$ to any $\pi_i$-subgroup $H$ takes the form $b\rho_H + a_i 1$ for some integer $b$. Also, every Brauer elementary subgroup is a $\pi_i$-group for some $i$ by our assumptions. Thus, $\chi$ is a generalized character of $G$ satisfying the desired property.
	\end{proof}
	
	We do not get a full converse to this theorem. Specifically, let $G$ be the semidirect product of the cyclic group of order 15, $C_{15}$, with the cyclic group of order 2, $C_2$, where the involution acts as the automorphism of inversion. The prime graph has two connected components since there exists a subgroup of order 15 and the elements of order two do not commute with the subgroup of order 15. Let $\lambda_1,\lambda_2 \in \Irr(C_{15})$ be the two nontrivial irreducible characters of $C_{15}$ whose kernels contain the cyclic group of order 5. Define $\chi(x) = \lambda_1(x) + \lambda_2(x) - 1$ for each $x \in C_{15}^{\#}$. We also take $\chi(x) = 0$ for each involution $x$. Lastly, define $\chi(1) = 16$. Brauer's Characterization of Characters implies that $\chi$ is a generalized character. This character is a sum of at most two roots of unity on each nonidentity element. In particular, there exists $x,y,z \in G^{\#}$ such that $\chi(x) = 0$, $\chi(y) = 1$, and $\chi(z) = -2$. Yet the prime graph has only two connected components: $\{2\}$ and $\{3,5\}$.
	
	The above paragraph leads us to the following question.
	
	\begin{question}\label{Q}
		Assume that $\chi \in \Z[\Irr(G)]$ and that the values of $\chi$ on $G^{\#}$ are sums of at most two roots of unity. Further assume that there exists $x,y,z \in G^{\#}$ such that $\chi(x) = 0$, $\chi(y)$ is a root of unity, and $\chi(z)$ is a sum of two roots of unity. Is the prime graph of $G$ guaranteed to be disconnected? 
	\end{question}
	\noindent Although we are unable to answer this question at the moment, we are however able to show in Theorem \ref{T:17} that under these circumstances, the graph becomes disconnected when the vertex corresponding to the prime 2 is thrown away.
	
	\section{Preliminaries}
	
	We begin by stating (without proof) a standard fact on sums of roots of unity.
	
	\begin{lemma}\label{L:8}
		Assume that $\sum_{i=1}^p \epsilon_i = 0$ for some prime $p$ and some $p$-power roots of unity $\epsilon_i$. Then, there exists a root of unity $\epsilon$ such that $\epsilon_i = \epsilon\zeta_{p}^i$ for each $i$. Also, if any sum of $p$-power roots of unity having $n$ terms is equal to zero, then $p \mid n$, and the terms of the sum can be partitioned into $n$ vanishing subsums having $p$ terms.
	\end{lemma}
	
	We call a vanishing sum of roots of unity \emph{minimal} if the terms of the sum cannot be partitioned into proper subsums which are themselves vanishing. The above lemma essentially says that the only minimal vanishing sums consisting of $p$-power roots of unity are the ``rotations'' of $\sum_{i=0}^{p-1}\zeta_{p}^i$ by some root of unity $\epsilon$. A classical result of de Bruijn in \cite{Bru} provides us with an analogous result for vanishing sums when two primes are involved. Alternatively, see \cite[Theorem 3.3]{LamLe} for a modern ring-theoretic proof.
	
	\begin{lemma}\label{L:9}
		Assume $\sum_{i=1}^n \epsilon_i$ is a minimal vanishing sum of $p^aq^b$-th roots of unity $\epsilon_i$ for distinct primes $p$ and $q$ and nonnegative integers $a$ and $b$. Then, $n$ is equal to either $p$ or $q$, and the sum is of the form in Lemma \ref{L:8}.
	\end{lemma} 
	In particular, this means that the terms of any vanishing sum of $p^aq^b$-th roots of unity may be partitioned into minimal vanishing subsums, each of which take either the form
	\begin{equation*}
		\sum_{i=0}^{p-1} \epsilon\zeta_p^i \;\;\; \mbox{or} \;\;\; \sum_{i=0}^{q-1} \epsilon\zeta_q^i
	\end{equation*}
	for some root of unity $\epsilon$.
	
	The above lemma has the following application to characters.
	
	\begin{proposition}\label{P:10}
		Let $\chi$ be any generalized character of a group $G$, and let $P$ be a nontrivial Sylow $p$-subgroup of $G$ for some prime $p$.
		\begin{itemize}
			\item[(i)] If $\chi(x)$ is a root of unity for some $p$-element $x$, then $\chi(1) \equiv \pm 1 \Mod{p}$.
			
			\item[(ii)] If $\chi$ is a root of unity for all nonidentity $p$-elements, then when $p$ is odd or $|P| = 2$, $\chi(1) \equiv \pm 1 \Mod{|P|}$. Otherwise, $p = 2$, $|P| \geq 4$, and $\chi(1) \equiv \pm 1 \Mod{|P|/2}$.
		\end{itemize}
	\end{proposition}
	
	\begin{proof}
		By adding integer multiples of the regular character to $\chi$, we may assume that $\chi$ is in fact a character of $G$ without altering our assumptions. Thus, we will assume $\chi$ is a character below.
		
		For (i), set $\chi(x) = \epsilon$ for some root of unity $\epsilon$ and some $p$-element $x$ so that $\chi(1) - \epsilon = 0$ is a vanishing sum of $2|P|$-th roots of unity having $\chi(1) + 1$ terms. Now, $\chi(x)$ is a sum of $\chi(1)$ roots of unity having $p$-power order. If $\epsilon = -\zeta_{p^a}^i$ for some integers $a$ and $i$, then $\chi(x) - \epsilon$ consists solely of terms having $p$-power order since $\chi$ is assumed to be a character. Thus, $p \mid \chi(1) + 1$ by Lemma \ref{L:8}. Thus, $\chi(1) \equiv -1 \Mod{p}$ in this case. 
		
		On the other hand, if $\epsilon = \zeta_{p^a}^i$ for integers $a$ and $i$, then $\chi(x) - \epsilon$ is a vanishing sum whose terms involve the primes 2 and $p$. (Notice that the minus sign on $\epsilon$ forces $\epsilon$ to have even order.) If $p = 2$, then Lemma \ref{L:8} applies as in the above paragraph. Otherwise, $p$ is odd, and Lemma \ref{L:9} applies. If $\chi(x) - \epsilon$ is partitioned into minimal vanishing subsums, the $-\epsilon$ term belongs to a subsum of two terms, taking the form $\epsilon - \epsilon$. All remaining terms are those coming from the $\chi(1)$ terms of $\chi(x)$, which each have $p$-power order since $\chi$ is a character. These remaining terms form a vanishing sum, the number of whose terms are divisible by $p$ by Lemma \ref{L:8}. It now follows that $p \mid \chi(1) - 1$, and so, $\chi(1) \equiv 1 \Mod{p}$. This proves (i).
		
		For (ii), notice that $\chi\ol{\chi} - 1$ is a character taking the value 0 on all nonidentity $p$-elements. Thus, the restriction of $\chi\ol{\chi} - 1$ to $P$ is an integer multiple of the regular $\rho_P$. By evaluating $\chi$ at the identity, we have now that $\chi(1)^2 \equiv 1 \Mod{|P|}$. From this fact, the above statement now follows.
	\end{proof}

	Our proof of the main theorem will involve a case analysis where we consider vanishing sums having up to six terms. We state a result (without proof) which will help us in that regard. First, we introduce a piece of terminology. If a sum of roots of unity has $k$ terms, we say that the sum has \emph{weight} $k$.
	\begin{lemma}\label{L:11}
		\begin{itemize}
			\item[(i)] There are no minimal vanishing sums of weight 4.
			\item[(ii)] The sum $\zeta_6 + \zeta_6^2 + \zeta_5 + \zeta_5^2 + \zeta_5^3 + \zeta_5^4$ is the unique minimal vanishing sum of weight 6 up to a rotation by some root of unity. 
		\end{itemize}
	\end{lemma}
	
	This lemma combined with Lemma \ref{L:8} describes all minimal vanishing sums of weight at most 6. Minimal vanishing sums have been studied extensively by many different authors for several different purposes, and those of weight up to 12 have essentially been classified (cf. \cite{PR}). An especially good source on this topic may be found in \cite{LamLe}.
	
	At last, we prove a lemma on characters of abelian groups.
	
	\begin{lemma}\label{L:12}
		Let $G$ be an abelian group of odd order, and assume that $\lambda_1, \lambda_2, \lambda_3 \in \Irr(G)$ are mutually distinct. There then exists $g \in G$ such that $\lambda_i(g) \neq \lambda_j(g)$ for $i \neq j$.
	\end{lemma}
	
	\begin{proof}
		By multiplying each character by $\ol{\lambda}_3$, we can assume without loss that $\lambda_3$ is trivial. It now suffices to produce an element $g \in G - (\ker\lambda_1 \cup \ker\lambda_2)$ (which is guaranteed to be nonempty) for which $\lambda_1(g) \neq \lambda_2(g)$. If $\ker\lambda_1 \subseteq \ker\lambda_2$ or vice versa, the statement follows by the distinctness of the two characters, so assume otherwise. In particular, we may assume that neither of the two characters are faithful.
		
		Set $H_i = \ker\lambda_i$ for $i=1,2$. Clearly, we may assume that $H_1 \cap H_2 = 1$ by taking quotients. Assume the theorem is false so that $\lambda_1(g) = \lambda_2(g)$ for all $g \in G - (H_1 \cup H_2)$. Notice that since $H_1$ is not a subgroup of $H_2$ and vice versa, $\lambda_i$ restricts nontrivially to $H_j$ for $i \neq j$. Thus, $[\lambda_1,\lambda_2]_{H_i} = 0$ for both $i$. We now partition $G$ as 
		$$G = 1 \cup H_1^{\#} \sqcup H_2^{\#} \sqcup G - (H_1 \cup H_2).$$ By the orthogonality relations, we obtain
		$$
		0 = |G|[\lambda_1,\lambda_2]_G = 1 + |H_1|[\lambda_1,\lambda_2]_{H_1} - 1 + |H_2|[\lambda_1,\lambda_2]_{H_2} - 1 + |G| - |H_1 \cup H_2|
		$$ $$
		= 1 - 1 - 1 + |G| - |H_1| - |H_2| + 1 = |G| - |H_1| - |H_2|.
		$$
		We obtain that $|H_1| + |H_2| = |G| \geq |H_1||H_2|$. Since neither $|H_1|$ nor $|H_2|$ equals 1, this forces $|H_1| = |H_2| = 2$, contradicting the fact that $|G|$ is odd.
	\end{proof}

	\section{Proofs of Main Result}

		Assume that $\chi \in \Z[\Irr(G)]$ satisfies the conditions of the main theorem. We may subtract rational (or even real) multiples of the regular character from $\chi$ without affecting the values $\chi$ takes on $G^{\#}$. Thus, we may assume that $\chi(1) = 0$. Now, we may write
		$$
			\chi = \sum_{\lambda \in \Irr(G)} a_\lambda \lambda
		$$
		for $a_\lambda \in \Q$. Note that although the coefficients $a_\lambda$ might not be integers, any two of these coefficients must differ by an integer. For two $\eta, \theta \in \Irr(G)$, we claim that $|a_\eta - a_\theta| \leq 2$. By the triangle inequality, we obtain that
		\begin{equation}\label{E:Temp}
			[\chi,\chi] = \sum_{\lambda \in \Irr(G)} a_\lambda^2 \leq \dfrac{4(|G|-1)}{|G|} < 4.
		\end{equation}
		Assume there exists $\eta, \theta \in \Irr(G)$ such that $a_\eta - a_\theta \geq 3$. We have that
	$$
		a_\theta^2 + (a_\theta \pm 3)^2 \leq a_\theta^2 + a_\eta^2 \leq [\chi,\chi] < 4.
	$$
		However, by treating $a_\theta^2 + (a_\theta \pm 3)^2$ as a polynomial in the variable $a_\theta$, one sees that the minimum value taken by this expression is $9/2$, contradicting the above inequality. Thus, there exists $x \in \Q$ such that $[\chi,\lambda] = x + i$, for $i = 0,1,2$ and for each $\lambda \in \Irr(G)$.

		 Now, let $a_i$ denote the number of $\lambda \in \Irr(G)$ such that $[\chi,\lambda] = x + i$. We may rewrite \eqref{E:Temp} in the following form:
		
		$$
			[\chi,\chi] = a_0x^2 + a_1(x+1)^2 + a_2(x+2)^2 \leq \dfrac{4(|G|-1)}{|G|}.
		$$
		Using the fact that $a_0 + a_1 + a_2 = |G|$, we obtain
		\begin{equation}\label{main_ineq}
			a_0x^2 + a_1(x+1)^2 + a_2(x+2)^2 \leq \dfrac{4(a_0 + a_1 + a_2 - 1)}{a_0 + a_1 + a_3}.
		\end{equation}
		The above inequality will be fundamental. When viewing this as a quadratic inequality in $x$, the discriminant is given by
		$$
			4(a_1 + 2a_2)^2 - 4(a_0+a_1+a_2)\bigg(a_1 + 4a_2 - \dfrac{4(a_0 + a_1 + a_2 - 1)}{a_0 + a_1 + a_2}\bigg).
		$$
		A valid choice of $x$ will not exist if this discriminant is negative, so we assume this not to be the case. Algebraically rearranging, we obtain
		$$
			4 \geq \dfrac{a_0a_1 + a_1a_2 + 4a_0a_2}{a_0 + a_1 + a_2 - 1}.
		$$
		If either $a_0$ or $a_2$ equals one, then one verifies that this inequality forces the other to be at most 3. Notice that if $a_0 = 1$ and $a_2 = 3$ (or vice versa), equality actually holds. Likewise, if $a_0$ and $a_2$ are both at least 2, then the above inequality fails, so this case cannot hold either. On the other hand, if $a_2 = 0$, then the above inequality becomes
		\begin{equation}\label{dumb-case}
			4(a_0 + a_1 - 1) - a_0a_1 \geq 0.
		\end{equation}
		One verifies that $a_0 \geq 5$ implies $a_1 \leq 16$, and by symmetry, $a_1 \geq 5$ implies $a_0 \leq 16$. Since $|G| = a_0 + a_1 + a_2$, this situation forces $|G| \leq 21$. We obtain the same result by symmetry if we take $a_0 = 0$. The case where $a_0 \geq 5$ will be handled in Case 5 below.
		
		We now have that either $|G| \leq 21$ or $\chi$ is of the form in the four outlier cases. We continue by cases from here. We will make repeated reference to Lemma \ref{L:8}.\\
		
		\paragraph{Case 1: $a_0 = 3$ and $a_2 = 0$, or vice versa.} In this case, $\chi = a\rho \pm (\lambda_1 + \lambda_2 + \lambda_3)$ for distinct $\lambda_i \in \Irr(G)$. Since the sign of $a$ is irrelevant, it is of no loss to multiply $\chi$ by a sign, so we may assume $\chi = a\rho + \lambda_1 + \lambda_2 + \lambda_3.$ Now, choose $g \in G^\#$. We have that
		$$
		\chi(g) = \lambda_1(g) + \lambda_2(g) + \lambda_3(g) = \delta_1\epsilon_1 + \delta_2\epsilon_2 
		$$
		for roots of unity $\epsilon_i$ and $\delta_i \in \{-1,0,1\}$. This gives us a vanishing sum having at most 5 terms:
		$$
		\lambda_1(g) + \lambda_2(g) + \lambda_3(g) - \delta_1\epsilon_1 - \delta_2\epsilon_2 = 0.
		$$
		
		If both $\delta_1$ and $\delta_2$ are nonzero, then we have a vanishing sum of weight 5 which is either a minimal vanishing sum of 5 terms or else is able to be partitioned into two minimal vanishing sums with one having 2 terms and the other having 3. In the former case, $\lambda_i(g)$ is a root of unity having order divisible by 5 for at least two $i$, and so $5 \mid o(g)$ and $5 \in \pi(G)$. In the latter case, if neither $\lambda_i(g)$ has order divisible by $3$, then this implies that $\lambda_i(g) = -\lambda_j(g)$ for some $i \neq j$, but this implies that $g$ has even order and $2 \in \pi(G)$. On the other hand, if any $\lambda_i(g)$ has order divisible by 3, then $3 \mid o(g)$ and $3 \in \pi(G)$.
		
		Otherwise, if $\delta_1 = \delta_2 = 0$, then $\lambda_1(g) + \lambda_2(g) + \lambda_3(g) = 0$ is a minimal vanishing sum of weight three, implying that $3 \mid o(g)$ and that $3 \in \pi(G)$.
		
		Lastly, if at exactly one of $\delta_1 = \pm 1$ and $\delta_2 = 0$ (or vice versa), then we have $\lambda_1(g) + \lambda_2(g) + \lambda_3(g) \pm \epsilon_1 = 0$, a vanishing sum of weight 4. This can only be partitioned into two vanishing sums of weight 2 by Lemma \ref{L:11}, and in this case, at least two of the terms will have even order. If $o(g)$ is odd, then each $\lambda_i$ will have odd order, an impossibility. Thus, $o(g)$ is even, and $2 \in \pi(G)$.\\
		
		\paragraph{Case 2: $a_0 = 2$ and $a_2 = 1$, or vice versa.} Again, we may assume without loss that $\chi = a\rho + \lambda_1 + \lambda_2 - \lambda_3$ for distinct $\lambda_i \in \Irr(G)$. Again, choose $g \in G^\#$. We may write
		$$
		\lambda_1(g) + \lambda_2(g) - \lambda_3(g) - \delta_1\epsilon_1 - \delta_2\epsilon_2 = 0 
		$$
		for roots of unity $\epsilon_i$ and $\delta_i \in \{-1,0,1\}$. If $\delta_1$ and $\delta_2$ are both nonzero, then this again gives us a vanishing sum of weight 5. If the vanishing sum is minimal, then it follows as before that $5 \in \pi(G)$. Otherwise, the sum may be partitioned into two minimal vanishing sums of weight 2 and 3. Now, at least two of the terms must have order divisible by 3. If one of these terms is $\lambda_i(g)$, then it follows as before that $3 \in \pi(G)$. Otherwise, it follows that $- \delta_1\epsilon_1 - \delta_2\epsilon_2 = \epsilon(\zeta_3 + \zeta_3^2) = -\epsilon$ for some root of unity $\epsilon$, so in particular, $\chi(g)$ is expressible as a root of unity, a case which we handle below.
		
		If $\delta_1 = \delta_2 = 0$, then it follows that $\lambda_1(g) + \lambda_2(g) - \lambda_3(g) = 0$ is a minimal vanishing sum of weight 3, and it again follows that $3 \in \pi(G)$.
		
		Otherwise, if $\delta_1 = \pm 1$ and $\delta_2 = 0$ (or vice versa), then $\lambda_1(g) + \lambda_2(g) - \lambda_3(g) - \delta_1\epsilon_1 = 0$ is a vanishing sum of weight 4. Assume $G$ is odd. By Lemma \ref{L:12}, for any three distinct linear characters of $G$, we can always choose a group element for which the $\lambda_i$ take mutually distinct values, so in particular, since $g$ has odd order, $\lambda_1(g) + \lambda_2(g) - \lambda_3(g)$ has no subsum which vanishes, which then implies that $\lambda_1(g) + \lambda_2(g) - \lambda_3(g) - \delta_1\epsilon_1 = 0$ is a minimal vanishing sum of weight 4, an impossibility by Lemma \ref{L:11}.\\
		
		\paragraph{Case 3: $a_0 = 0$ and $a_2 = 4$, or vice versa.} Again, we may assume $\chi(g) = a\rho + \lambda_1 + \lambda_2+ \lambda_3 + \lambda_4$. We now have that
		$$
		\lambda_1(g) + \lambda_2(g) + \lambda_3(g) + \lambda_4(g) - \delta_1\epsilon_1 - \delta_2\epsilon_2 = 0.
		$$
		If the $\delta_i$'s are both nonzero, we have a vanishing sum of weight 6, which, if not minimal, may be partitioned into either two minimal vanishing sums of weight three or three minimal vanishing sums of weight 2. In the former case, at least four of the terms must have order divisible by 3, and it follows that at least one of the $\lambda_i(g)$'s has order divisible by three, so in this case, $3 \in \pi(G)$. In the latter case, at least three of the terms must have even order, yet if $g$ is chosen to be odd, then this also implies that each of the four $\lambda_i$'s have odd order, an impossibility. So in this case, $2 \in \pi(G)$. If the vanishing sum is in fact minimal, then by Lemma \ref{L:11}, the sum is of the form 
		$$
		\epsilon\zeta_6 + 	\epsilon\zeta_6^2 + 	\epsilon\zeta_5 + 	\epsilon\zeta_5^2 + 	\epsilon\zeta_5^3 + 	\epsilon\zeta_5^4
		$$
		for some root of unity $\epsilon$. One sees that regardless of the choice of $\epsilon$, at least 3 terms of the sum must have order divisible by 5. Thus, at least one of the $\lambda_i(g)$'s has order divisible by 5, so $5 \in \pi(G)$.
		
		If both $\delta_i$'s are zero, then we have that $\lambda_1(g) + \lambda_2(g) + \lambda_3(g) + \lambda_4(g) = 0$ is a vanishing sum of weight 4. If $g$ is chosen to have odd order, then each terms has odd order, impossible for a vanishing sum of weight 4 by Lemma \ref{L:11}. Thus, $2 \mid o(g)$ and $2 \in \pi(G)$ in this case.
		
		If $\delta_1 = \pm 1$ and $\delta_2 = 0$ (or vice versa), then we have a minimal vanishing sum of weight 5: $\lambda_1(g) + \lambda_2(g) + \lambda_3(g) + \lambda_4(g) - \delta_1\epsilon_1 = 0$. It follows then as above that at least one of the $\lambda_i(g)$'s has order divisible by 3 or 5, so either $3 \in \pi(G)$ or $5 \in \pi(G)$.\\
		
		\paragraph{Case 4: $a_0 = 1$ and $a_2 = 3$, or vice versa.} Again, we may assume $\chi(g) = a\rho + \lambda_1 + \lambda_2+ \lambda_3 - \lambda_4$. We now have that
		$$
		\lambda_1(g) + \lambda_2(g) + \lambda_3(g) - \lambda_4(g) - \delta_1\epsilon_1 - \delta_2\epsilon_2 = 0.
		$$
		
		Notice that in this case, we have equality in \eqref{main_ineq}, and it follows by the triangle inequality that $\delta_1\epsilon_1 = \delta_2\epsilon_2$. In particular, this vanishing sum cannot be minimal since the terms of a minimal vanishing sum of this weight are each distinct by Lemma \ref{L:11}. If the sum can be partitioned into two minimal vanishing sums of weight 3, then it follows as in Case 3 that at least one of the $\lambda_i(g)$'s has order divisible by $3$, and so, $3 \in \pi(G)$. On the other hand, we may have that the sum can be partitioned into three minimal vanishing sums of weight 2. If $G$ has odd order, then we can choose $g \in G$ such that the first three $\lambda_i$ take mutually distinct values on $g$ by Lemma \ref{L:12}. If $\lambda_4(g) = \lambda_i(g)$ for some $i \neq 4$ (say $i = 3$ for instance), then this implies $\lambda_1(g) + \lambda_2(g) - \pm 2\epsilon_1 = 0$, and thus, $\lambda_1(g) = \lambda_2(g)$, a contradiction. On the other hand, if $\lambda_4(g) \neq \lambda_i(g)$ for each $i$, then $\lambda_4(g)$ cancels with one of the $\delta_i\epsilon_i$ terms. This gives rise to a vanishing sum of weight 4 having three terms of odd order, which is impossible.\\
		
		\paragraph{Case 5: $a_0 \geq 5$ or $a_2 \geq 5$.} We have seen that $a_0 \geq 5$ implies $a_2 = 0$ and vice versa. Thus, we may assume without loss that $a_2 = 0$ and $a_1 \geq 5$. We have seen that this implies $|G| \leq 21$. The only groups of order that are at most 21 and are indivisible by either 2, 3, or 5 are the cyclic groups of order 7, 11, 13, 17, and 19. Further analysis of the inequality in \eqref{dumb-case} shows that $a_1 = 6$ implies $a_0 \leq 10$ and that $a_1 = 7$ implies $a_0 \leq 8$. Cases with higher values of $a_1$ are subsumed by the cases considered here by symmetry. In particular, it follows that in this case, $\chi = a\rho_G \pm \sum_{i=1}^k \lambda_i$ for distinct $\lambda_i \in \Irr(G)$ and $5 \leq k \leq 7$.
		
		If $a_1 = 5$, then considering $\chi(g)$ for any element $g \in G^{\#}$ gives rise to a vanishing sum of weight between 5 and 7. Since over half of the terms will have odd order, it follows by Lemma \ref{L:9} that the terms of $\chi(g)$ are 7-th roots of unity, and so $|G| = 7$. In this case however, $\chi$ is some multiple of the regular character minus two linear characters. Similarly, it follows that $|G| = 7$ when considering the cases where $a_1$ is 6 or 7, meaning that $\chi$ is not a meaningful counterexample. This completes the proof.
		
	We have actually proven a few other facts within this proof, which we  state now for future reference.
	\begin{corollary}\label{C:13}
		Assume that $G$ is an abelian group having a generalized character $\chi$ whose values on $G^{\#}$ are sums of at most two roots of unity, and assume that $\chi = a\rho_G \pm \sum_{i=1}^k \lambda_i$ for $a \in \Z$ and distinct $\lambda_i \in \Irr(G)$. If $k = 6$, then $|G| \leq 16$. If $k = 7$, then $|G| \leq 15$. 
		
		Also, assume that either of the following holds:
		\begin{itemize}
			\item[(i)] $\chi = a\rho_G \pm (\lambda_1 + \lambda_2 + \lambda_3 - \lambda_4)$ for $a \in \Z$ and distinct $\lambda_i \in \Irr(G)$,
			\item[(ii)] $|G| = 21$ and $\chi = a\rho_G \pm \sum_{i=1}^5 \lambda_i$ for $a \in \Z$ and distinct $\lambda_i \in \Irr(G)$,
			\item[(iii)] $|G| = 16$ and $\chi = a\rho_G \pm \sum_{i=1}^6 \lambda_i$ for $a \in \Z$ and distinct $\lambda_i \in \Irr(G)$, or
			\item[(ii)] $|G| = 15$ and $\chi = a\rho_G \pm \sum_{i=1}^7 \lambda_i$ for $a \in \Z$ and distinct $\lambda_i \in \Irr(G)$.
		\end{itemize}
		Then for each $x \in G^{\#}$, we have that $\chi(x) = 2\epsilon$ where $\epsilon$ is a root of unity.
	\end{corollary}

	We wonder more generally about generalized characters of abelian groups whose values on nonidentity elements are sums of at most $k$ roots of unity. However, the case analysis required here indicates that it may be difficult to obtain a meaningful result in this level of generality.
		
	\section{Prime Graphs}

	 We now turn to proving a result on prime graphs as an application of our main theorem. For a (not necessarily abelian) group $G$, $\Gamma(G)$ is taken to be the prime graph of $G$.

	\begin{lemma}\label{L:14}
		Let $G$ be a group, and assume that $\chi \in \Z[\Irr(G)]$ takes values on $G^{\#}$ which are sums of at most two roots of unity. Then for any $p \in \pi(G)$, we have that $\chi(1) \equiv a \Mod{p}$ where $a \in \{-2,-1,0,1,2\}$.
	\end{lemma}
	
	\begin{proof}
		Clearly the result holds if $p \leq 5$, so assume otherwise. Choose an abelian subgroup $H$ of some $P \in \Syl_p(G)$. By our main theorem, the restriction of $\chi$ to $H$ takes the form $\chi_H = \rho_H + \delta_1\lambda_1 + \delta_2\lambda_2$ for $\delta_i \in \{-1,0,1\}$ and $\lambda_i \in \Irr(H)$. The result now follows. 
	\end{proof}
	
	Under our assumptions on $G$, this result allows us to partition the primes of $\pi(G) - \{2,3\}$ into 5 sets: $\bigsqcup_{i=-2}^2 \pi_i$, where $p \in \pi_i$ if $\chi(1) \equiv i \Mod{p}$. We obtain the following as a consequence.
	
	\begin{theorem}\label{T:15}
		Let $G$ be a group, and assume $\chi \in \Z[\Irr(G)]$ takes values on $G^{\#}$ which are sums of at most two roots of unity. Also, let $\Gamma''(G)$ denote the induced subgraph of $\Gamma(G)$ on the set $\pi(G) - \{2,3\}$, and let $\Gamma_i(G)$ denote the induced subgraph of $\Gamma(G)$ on the vertex set $\pi_i$. Then, if $\pi_i$ is nonempty, $\Gamma_i(G)$ is a union of connected components of $\Gamma''(G)$. 
	\end{theorem}
	
	\begin{proof}
		First, assume that $p$ and $q$ are primes such that $5 \leq p < q$. Further assume that if $G$ has an an abelian subgroup $H$ of order $pq$, then the restriction $\chi_H$ is of the form $a\rho_H + \delta_1\lambda_1 + \delta_2\lambda_2$ for $\lambda_i \in \Irr(H)$, $a \in \Z$, and $\delta_i \in \{-1,0,1\}$. In particular, this covers the case where $7 \leq p < q$ by our main result. Let $H_p$ and $H_q$ be the Sylow $p$- and Sylow $q$-subgroups of $H$ respectively. Now, the restriction of the regular character to any subgroup is a multiple of the regular character of that subgroup, so the restriction of $\chi$ to $H_p$ takes the form $qa\rho_{H_p} + \delta_1(\lambda_1)_{H_p} + \delta_2(\lambda_2)_{H_p}$. Similarly, the restriction of $\chi$ to $H_q$ is $pa\rho_{H_q} + \delta_1(\lambda_1)_{H_q} + \delta_2(\lambda_2)_{H_q}$. By evaluating $\chi$ at 1, we obtain that $\delta_1 + \delta_2 \equiv \chi(1) \Mod{p}$ and $\Mod{q}$. Since $\delta_1 + \delta_2 \in \{-2,-1,0,1,2\}$, it follows that $p,q \in \pi_{\delta_1+\delta_2}$. Thus in this case, there is no path in $\Gamma''(G)$ connecting an element of $\pi_i$ to an element of $\pi_j$ for $i \neq j$.
		
		Now, assume there exists an abelian subgroup $H$ of order $pq$ such that $\chi_H$ is not of the form $a\rho_H + \delta_1\lambda_1 + \delta_2\lambda_2$ for $\lambda_i \in \Irr(H)$, $a \in \Z$, and $\delta_i \in \{-1,0,1\}$. In particular, our main result implies that $p = 5$. Further , since $|H| = 5q > 21$, $\chi_H$ is of the form of one of the four outlier cases. Set $H_5$ and $H_q$ as the Sylow $5$- and Sylow $q$-subgroups respectively. We proceed by cases.\\
		
		\paragraph{Case 1: $\chi_H = a\rho_H \pm \lambda_1 \pm \lambda_2 \pm \lambda_3$ or $\chi_H = a\rho_H \pm \lambda_1 \pm \lambda_2 \pm \lambda_3 \pm \lambda_4$} Assume first that $\chi_H = a\rho_H + \lambda_1 + \lambda_2 + \lambda_3$. The case where the $\lambda_i$'s are each being subtracted will follow by symmetry. By the main result, the $\lambda_i$ must remain distinct when restricting $\chi$ to $H_q$. In particular, since $q > 3$, none of the $\lambda_i$'s can be absorbed into the regular character of $H_q$. However, evaluating $\chi$ at $x \in H_q$ gives $\chi(x) = \lambda_1(x) + \lambda_2(x) + \lambda_3(x) = \delta_1\epsilon_1 + \delta_2\epsilon_2$ for $\delta_i \in \{0,1\}$ and $2q$-th roots of unity $\epsilon_1$. Taking into account the fact that the $\lambda_i(x)$ are $q$-th roots of unity, one verifies using Lemma \ref{L:9} that there is no way to partition the terms of $\chi(x) - \delta_1\epsilon_1 - \delta_2\epsilon_2$ into minimal vanishing sums. Thus, this case is impossible. The cases where $\chi_H = a\rho_H \pm \lambda_1 \pm \lambda_2 \pm \lambda_3 \pm \lambda_4$ are impossible as well by an analogous argument. \\
		
		\paragraph{Case 2: $\chi_H = a\rho_H \pm \lambda_1 \pm \lambda_2 \mp \lambda_3$ or $\chi_H= a\rho_H \pm \lambda_1 \pm \lambda_2 \pm \lambda_3 \mp \lambda_4$.} In the former cases, as in the first paragraph of this proof, we obtain that $\chi(1)  \equiv \pm 1 \Mod{5}$ and $\Mod{q}$. That is, either $p,q \in \pi_1$ or $p,q \in \pi_{-1}$. Likewise, in the latter cases, we have $\chi(1) \equiv \pm 2 \Mod{5}$ and $\Mod{q}$, and so, either $p,q \in \pi_2$ or $p,q \in \pi_{-2}$. \\
	\end{proof}
	
	We now may obtain a result on the connectivity of the induced subgraph of $\Gamma(G)$ on the vertex set $\pi(G) - \{2\}$. First, we introduce some notation. Let $H$ be an abelian group, and let $\chi$ be a generalized character of $H$ taking the form
	
	$$
	\chi = a\rho_H + \sum_{i=1}^k \eta_i - \sum_{i=1}^\ell \theta_i
	$$
	for $a \in \Z$ and where all of the $\eta_i$'s and $\theta_i$'s are mutually distinct. We say that $\chi$ has \emph{type} $(k,\ell)$. Our main theorem implies that if $H$ is an abelian group having a generalized character $\chi$ whose values on $H^{\#}$ are sums of at most two roots of unity, then $\chi$ has one of the twenty following types:
	
	$$
	\begin{matrix}
	(0,0) & (1,1) & (1,0) & (0,1) & (2,0) \\
	(0,2) & (2,1) & (1,2) & (3,1) & (1,3) \\
	(3,0) & (0,3) & (4,0) & (0,4) & (5,0) \\
	(0,5) & (6,0) & (0,6) & (7,0) & (0,7)
	\end{matrix}
	$$
	
	We first have a lemma which will aid in our case analysis.
	
	\begin{lemma}\label{L:16}
		 Let $p$ be a prime, and let $C$ denote the cyclic group of order $3p$.
		\begin{itemize}
			\item[(i)] If $p \geq 7$, then $C$ has no characters of the types $(3,0)$, $(0,3)$, $(4,0)$, or $(0,4)$.
			\item[(ii)] If $p = 5$, then $C$ has no generalized characters of the types $( 5, 0)$, $(0, 5)$, $( 6, 0)$, or $(0, 6)$.
		\end{itemize}
	\end{lemma}
	
	\begin{proof}
		For part (i), let $P$ be the cyclic subgroup of order $p$. In any of these four cases, one sees that our main theorem implies that when restricting $\chi$ to $P$, the $\lambda_i$ must remain distinct. In particular, since $p \geq 7$, none of 3 or four constituents can be absorbed into a regular character of $P$. But our main theorem also implies that this situation is only possible when $p \leq 5$.   
		
		The nonexistence of examples in (ii) was verified using GAP \cite{Gap}.
	\end{proof}
	
	We at last have our final result.
	
	\begin{theorem}\label{T:17}
		Let $G$ be a group, and assume there exists $\chi \in \Z[\Irr(G)]$ whose values on $G^{\#}$ are sums of at most two roots of unity. Let $\Gamma'(G)$ denote the induced subgraph of $\Gamma(G)$ obtained by discarding the vertex 2. Further, assume that the three sets $\pi_0$, $\pi_1 \cup \pi_{-1}$, and $\pi_2 \cup \pi_{-2}$ are each nonempty. Then, $\Gamma'(G)$ is disconnected. 
	\end{theorem}
	
	\begin{proof}
		Using the notation of the previous theorem, we take advantage of the fact that the induced subgraph of $\Gamma''(G)$ on the vertex set $\pi_i$ is either empty or a union of connected components of $\Gamma''(G)$. Thus, if any element of $\pi_i$ becomes connected by a path to an element of $\pi_j$ for $i \neq j$, this path must pass through the vertex 3. We proceed in cases depending on the congruence class of $\chi(1)$ modulo 3. In each case, we investigate the possible primes $p$ such that a cyclic subgroup of order $3p$ can exist within $G$.
		
		First, consider the case where $3 \mid \chi(1)$. If $H$ is an abelian subgroup of order $3p$ for $p \geq 5$, then $\chi_H$ can only have the following types:
		$$
			(0,0), (1,1), (3,0), (0,3), (6,0), (0,6).
		$$
		However, if $\chi_H$ has the latter two types, this implies $|H| \leq 16$ by Corollary \ref{C:13}, and so $p = 5$, but this is impossible by the above lemma. If $\chi_H$ has either of the first two types, this implies that 3 is linked to an element of $\pi_0$. If $G$ has another subgroup $K$ of order $3q$ for some prime $q$ such that $\chi_K$ has type $(3,0)$ or $(0,3)$, then $p = 5$ by our previous lemma. However, these two types cannot occur within $G$ simultaneously since this would imply that $\chi(1) \equiv -2 \Mod{5}$ and $\chi(1) \equiv 2 \Mod{5}$. In any case, 3 can at most link $\pi_0$ to either $\pi_{-2}$ or $\pi_{2}$, but not both simultaneously. More importantly, there is no path in $\Gamma'(G)$ linking an element of $\pi_1 \cup \pi_{-1}$ to an element of $\pi_0 \cup \pi_2 \cup \pi_{-2}$, and the result follows in this case.
		
		Now, assume that $\chi(1) \equiv 1 \Mod{3}$. We claim that the induced subgraph of $\Gamma(G)$ on $\pi_0$ is a union of connected components of $\Gamma'(G)$. By the previous theorem, it suffices to show that there exists no cyclic subgroup of order $3p$ for $p \in \pi_0$. If $H$ is an abelian subgroup of order $3p$, then $\chi_H$ can only have the following types:
		$$
		(1,0), (4,0), (7,0), (0,2), (1,3), (0,5), (2,1).
		$$
		If $\chi_H$ has type $(k, \ell)$ and $p \in \pi_0$, then $\chi(1) = a|H| + k - \ell$ is divisible by $p$, but since $p \mid |H|$, we have that $p \mid (k - \ell)$. Given the possible types $(k,\ell)$ above, the only prime divisors greater than 3 for any choice of $k - \ell$ are $5$ and $7$. If $p = 5$, then the only possible type for $\chi_H$ is $(0,5)$, but by Lemma \ref{L:16}, this is impossible since $|H| = 15$ in this case. On the other hand, if $p = 7$, then the only possible type for $\chi_H$ is $(7,0)$, but by Corollary \ref{C:13}, this type can only occur if $|H| \leq 15$ despite the fact that $|H| = 21$ in this case. Thus, $3$ is not not linked to any element of $\pi_0$ in $\Gamma'(G)$. However, by Theorem \ref{T:15}, there exist no edges between an element of $\pi_0$ and an element of $\pi_i$ for $i \neq 0$ in $\Gamma'(G)$.  Thus, the vertices of $\pi_0$ form a union of connected component of $\Gamma'(G)$ in this case.

		Lastly, if $\chi(1) \equiv 2 \Mod{3}$, then by replacing $\chi$ with $-\chi$, the generalized character now satisfies $\chi(1) \equiv 1 \Mod{3}$, so the previous case implies $\Gamma'(G)$ is disconnected in this case as well.
	\end{proof}
	
	We suspect that $\Gamma(G)$ itself is disconnected under the assumptions of this theorem. However, it appears that the techniques used here are not strong enough to prove such a result via a similar case analysis.

	

\begin{thebibliography}{1}
		\bibitem[1]{Bru} N.G. de Bruijn. On the Factorization of Cyclic Groups. \emph{Indag. Math}. \textbf{15} (1953), 370--377.
		
		\bibitem[2]{Gap} The GAP Group, GAP -- Groups, algorithms, and programming, Version 4.15.0 (2024). https://www.gap-system.org.
		
		\bibitem[3]{GK} K.W. Gruenberg and O. Kegel. Unpublished manuscript (1975).
		
		\bibitem[4]{HW} P. H\"ahndel and R. Waldecker. Corrigendum and addendum to Transitive permutation groups where nontrivial elements have at
		most two fixed points. Preprint (2025).
		
		\bibitem[5]{Is} I.M. Isaacs. \emph{Character Theory Of Finite Groups}, (Academic Press, 1976).
		
		\bibitem[6]{LamLe} T.Y. Lam and K.H. Leung. On vanishing sums of roots of unity, \emph{J. of Alg.} \textbf{224} (2000), 91--109.
		
		\bibitem[7]{PR} B. Poonen and M. Rubinstein. The Number of Intersection Points Made by the Diagonals of a Regular Polygon. \emph{SIAM J. of Discrete Math}. \textbf{11} (1998), 133--156.
		
		\bibitem[8]{R} G.R. Robinson. Generalized Characters whose values on non-identity elements are roots of unity, \emph{J. of Alg.}, \textbf{333} (2011), 458--464.	
		
		\bibitem[9]{R2} G.R. Robinson. A Bound on Norms of Generalized Characters with Applications, \emph{J. of Alg.} \textbf{212} (1999), 660--668
		
		\bibitem[10]{W} J.S. Williams. Prime Graph Components of Finite Groups, \emph{J. of Alg.} \textbf{69} (1981), 487--513.
		
	\end{thebibliography}
\end{document}